\documentclass[
12pt]{amsart}
\usepackage{amscd}
\usepackage{amssymb}
\usepackage{a4wide}
\usepackage{amstext}
\usepackage{amsthm}
\usepackage{mathrsfs}
\usepackage[final]{hyperref}
\hypersetup{unicode= false, colorlinks=true, linkcolor=blue,
anchorcolor=blue, citecolor=green, filecolor=red, menucolor=blue,
pagecolor=blue, urlcolor=blue} 
\numberwithin{equation}{section}

\newcommand{\Hg}{\mathcal{H}}

\DeclareMathOperator{\supp}{supp}

\renewcommand{\phi}{\varphi}
\newcommand{\rea}{{\rm Re}\,}
\newcommand{\ima}{{\rm Im}\,}

\theoremstyle{plain}
\newtheorem{theorem}{Theorem}[section]

\newtheorem{proposition}[theorem]{Proposition}
\newtheorem{lemma}[theorem]{Lemma}
\newtheorem{corollary}[theorem]{Corollary}

\newtheorem*{question}{Question}

\newtheorem{definition}[theorem]{Definition}
\newtheorem{example}[theorem]{Example }

\newtheorem{remark}[theorem]{Remark}

\begin{document}

\title[Systems of reproducing kernels and their biorthogonal]
{Systems of reproducing kernels and their biorthogonal:
completeness or incompleteness? }

\author [Anton Baranov]{Anton Baranov}
\address{ Department of Mathematics and Mechanics, Saint Petersburg
State Univerisity, 28, Universitetski pr., St. Petersburg, 198504,
Russia}
 \email{anton.d.baranov@gmail.com}
\author [Yurii Belov]{Yurii Belov }
\address{Department of Mathematical Sciences\\
Norwegian University of Science and Technology (NTNU)\\
 NO- 7491 Trondheim, Norway}
 \email{j\_b\_juri\_belov@mail.ru}

\thanks{This work was supported by
the Chebyshev Laboratory  
(Department of Mathematics and Mechanics, St. Petersburg State University)  
under RF government grant 11.G34.31.0026, by grant MK 7656.2010.1 
(Russia), and by the Research Council of Norway, grant 185359/V30.}

\begin{abstract}
Let $\{v_n\}$ be a complete minimal system in a Hilbert space
$\mathcal{H}$ and let $\{w_m\}$ be its biorthogonal system. It is
well known that $\{w_m\}$ is not necessarily complete. However the
situation may change if we consider systems of reproducing kernels
in a reproducing kernel Hilbert space $\mathcal{H}$ of analytic
functions. We study the completeness problem for a class of spaces
with a Riesz basis of reproducing kernels and for model subspaces
$K_\Theta$ of the Hardy space. We find  a class of spaces where
systems biorthogonal to complete systems of reproducing kernels are
always complete, and show that in general this is not true. In
particular we answer the question posed by N.K. Nikolski and
construct a model subspace  with an incomplete biorthogonal system.
\end{abstract}



\maketitle
\section{Introduction and main results}
\subsection{Statement of the problem}
Let $\mathcal{H}$ be a separable Hilbert space. A sequence of
vectors $\{v_n\}$ is said to be {\it complete} if
$\overline{Span\{v_n\}}=\mathcal{H}$. If, moreover, the system
$\{v_n\}$ is {\it minimal}, i.e. it
fails to be complete when we remove any vector, then we
say that the system is {\it exact}. For every exact
system of vectors $\{v_n\}$ there exists a unique biorthogonal
system $\{w_m\}$ such that $\langle v_n,w_m \rangle=\delta_{mn}$.

Suppose that $\mathcal{H}$ is a space of entire functions with
reproducing kernels. Namely, for each $\lambda \in\mathbb{C}$ there is an
element $k_\lambda \in\mathcal{H}$ such that
$\langle f,k_\lambda\rangle =f(\lambda)$ for all
$f\in\mathcal{H}$. We are looking for an answer to the following
question:

\begin{question}
Let $\{k_{\lambda_n}\}$ be an exact system of reproducing kernels in
$\mathcal{H}$. Is it true that its biorthogonal system is also
complete in $\mathcal{H}$?
\end{question}

Of course, for an arbitrary sequence of vectors its biorthogonal
system may be incomplete. If $\{e_n\}_{n=1}^\infty$ is an
orthonormal basis, then the system $\{e_n+e_1\}_{n=2}^\infty$ is
complete, but the biorthogonal system $\{e_n\}_{n=2}^{\infty}$ is
incomplete. On the other hand, it is well known that if we
restrict ourselves to {\it systems of reproducing kernels}, then the
answer may be positive. R.M. Young \cite{Y} proved the completeness
of systems biorthogonal to the systems of reproducing kernels in the
Paley--Wiener spaces or, equivalently, for biorthogonals to systems
of exponentials in $L^2$ on an interval. E. Fricain \cite{Fr}
extended this result to a class of de Branges spaces of entire
functions (see discussion below).

Our aim is to exhibit some classes of spaces for which we know the
answer (positive or negative). In particular, we answer the question
posed by N.K. Nikolski and construct an example of a model
(shift-coinvariant) subspace of the Hardy space $H^2$ with an
incomplete biorthogonal system.
In what follows we consider only systems biorthogonal 
to systems of reproducing kernels; therefore, sometimes we 
write simply {\it the biorthogonal system} in place 
of {\it the system biorthogonal to an exact 
system of reproducing kernels.}

To make the problem more realistic we need some additional
structure on $\mathcal{H}$, namely the existence of a \itshape Riesz
basis\normalfont. Recall that a system of vectors $\{v_n\}$ is said
to be a Riesz basis if $\{v_n\}$ is an image of an orthonormal basis
under a bounded and invertible operator in $\mathcal{H}$. We
consider the class $\mathfrak{R}$ of spaces of entire functions
satisfying three axioms:
\begin{itemize}
\item[(A1)] $\Hg$ has a
reproducing kernel $k_{\lambda}$ at every point
$\lambda\in\mathbb{C}$;
\item[(A2)] If a function $f$ is in $\mathcal{H}$ and $f(w)=0$, then
the function $\frac{f(z)}{z-w}$ is also in $\mathcal{H}$;
\item[(A3)] There exists a sequence of distinct points $T=\{t_n\}
\subset \mathbb{C}$
such that the sequence of normalized reproducing kernels
$\big\{k_{t_n}/\|k_{t_n}\|_{\Hg}\big\}$ is a Riesz basis for $\Hg$.
\end{itemize}

First example of such spaces is the Paley--Wiener space $PW_\pi$
which is the space of entire functions of exponential type at most
$\pi$ that are in $L^2(\mathbb{R})$. In this case the sequence
$\{\frac{\sin(\pi (z-n))}{\pi (z-n)}\}_{n\in\mathbb{Z}}$ is an
orthonormal basis of reproducing kernels and $(A3)$
is satisfied. Axioms $(A1), (A2)$ follow immediately.

More general examples are de Branges spaces.
We say that an entire function $E$ belongs to the Hermite--Biehler class
if it has no real zeros and $|E(z)|>|E(\overline{z})|$ for any $z$
in the upper half-plane $\mathbb{C}^+$.
The de Branges space $\mathcal H(E)$ consists of all
entire functions $f$ such that $f\slash E$ and
$f^*\slash{E}$ belong to the
Hardy space $H^2$ in $\mathbb{C}^+$;
here $f^*(z) = \overline{f(\overline{z})}$.
The norm in $\mathcal H(E)$ is given by
$$
    \|f\|^2_{\mathcal{H}(E)}=\int_{\mathbb{R}}\frac{|f(x)|^2}{|E(x)|^2}\,dx.
$$
As in the Paley--Wiener space, in de Branges spaces there exist
orthonormal bases of reproducing kernels (see \cite{DB}). So de
Branges spaces form a subclass of our class $\mathfrak{R}$.

\subsection{Parametrization of the class $\mathfrak{R}$}
We will use an explicit parametrization of the class $\mathfrak{R}$
from \cite{BMS}. Let us briefly remind the description from
\cite{BMS}.

The Riesz basis $\{k_{t_n}/\|k_{t_n}\|_{\Hg}\}$ has a biorthogonal
basis, which we will call $\{f_n\}$. Using axiom (A2) we
conclude that $f_n(z)\frac{(z-t_n)}{(z-t_m)}\in\mathcal{H}, m\neq
n$. This function vanishes at the points $t_l$, $l\neq m$, and so it
equals $f_m$ up to a multiplicative constant. Hence, the function $c_m
f_m(z)(z-t_m)$ does not depend on $m$ for suitable coefficients
$c_m$. We call this function the {\it generating function} of the
sequence $\{t_n\}$ and denote it by $F$. Note that
using this construction we may define the generating function 
for an arbitrary exact system of reproducing kernels, 
not just for a Riesz basis.

The sequence $f_n$ is also
a Riesz basis for $\Hg$, and therefore any vector $h$ in $\Hg$ can
be written as
\begin{equation}
\label{lagrange}
h(z)=\sum_{n}
h(t_n)\frac{F(z)}{F'(t_n)(z-t_n)},
\end{equation}
where the sum converges with respect to the norm of $\Hg$ and
\[
A\sum_n \frac{|h(t_n)|^2}{\|k_{t_n}\|_{\Hg}^2}
\le \|h\|_{\Hg}^2  \le B \sum_n \frac{|h(t_n)|^2}{\|k_{t_n}\|_{\Hg}^2}
\]
for some constants $A,B>0$ independent of $h$.
Since point evaluation at every point $z$ is a bounded
linear functional, the series in \eqref{lagrange} also converges 
uniformly on compact subsets of $\mathbb{C}\setminus T$. 
By the assumption that $h\mapsto
\{h(t_n)/\|k_{t_n}\|_{\Hg}\}$ is a bijective map from $\Hg$ to
$\ell^2$, we get
\begin{equation}
\label{pointwise}
\sum_n \frac{\|k_{t_n}\|_{\Hg}^2}{|F'(t_n)|^2|z-t_n|^2}<+\infty
\end{equation}
whenever $z$ is in $\mathbb{C}\setminus T$. Therefore
\eqref{pointwise} implies that
\begin{equation}
\sum_{n} \frac{b_n}{1+|t_n|^2}<+\infty,\qquad
b_n:=\frac{\|k_{t_n}\|_{\Hg}^2}{|F'(t_n)|^2}.
\label{adm}
\end{equation}

It follows from \eqref{lagrange} that we can associate with the space
$\Hg\in\mathfrak{R}$ a space of \itshape meromorphic functions
with prescribed poles\normalfont. Namely, given a sequence of
distinct complex numbers $T=\{t_n\}$ and a weight sequence $b=\{b_n\}$
which satisfy the admissibility condition \eqref{adm}, we introduce
the space $\Hg(T,b)$ consisting of all functions of the form
\begin{equation}
f(z)=\sum_{n=1}^\infty\frac{a_n b^{1/2}_n}{z-t_n} \label{param}
\end{equation}
such that
\[
\|f\|_{\Hg(T,b)}^2:=\sum_{n=1}^\infty |a_n|^2 <+\infty.
\]

The map $f\mapsto Ff$ is an isomorphism of $\Hg(T,b)$ onto $\Hg$
which maps reproducing kernels to reproducing kernels. So, for our
approach, we can consider the pairs $(T,b)$ as a parametrization of
all spaces from $\mathfrak{R}$.

The space $\Hg \in \mathfrak{R}$ is a de Branges space $\Hg(E)$ for
some Hermite--Biehler function $E$ if and only if there exists $T$
such that $T\subset\mathbb{R}$. In this case we may choose $E$ so
that $F=\frac{E+E^*}{2}$. Moreover, as was shown in \cite{BMS2}, de
Branges spaces are the only spaces of the class $\mathfrak{R}$ where
there exist two different orthogonal bases of reproducing kernels.
These spaces are the prime example for us. Note, however, that there
are many spaces in the class $\mathfrak{R}$ which are not isomorphic
to a de Branges space. E.g., let $T = \{u_n\} \cup \{i w_n\}$, where
$u_n$, $w_n$ are arbitrary sequences of real points satisfying
$\sum_n |u_n|^{-1} = \sum_n |w_n|^{-1} =\infty$. Then the space
$F\Hg(T,b)$ is not a de Branges space (in any half-plane) since for
a Riesz sequence of normalized kernels
$\{k_{\lambda_n}/\|k_{\lambda_n}\|_{\Hg}\}$ in a de Branges space
$\Hg$, the sequences $\{\lambda_n\} \cap \mathbb{C}^+$ and
$\{\lambda_n\} \cap \mathbb{C}^-$ should satisfy the Carleson
interpolation condition \cite[Part D, Lemma 4.4.2]{nk2}.

\subsection{Main theorems}
Now we are ready to state our main results.
\begin{theorem}
\label{main1}
If $\Hg\in\mathfrak{R}$ and
$\sum_n b_n<+\infty$,
then there exists an exact system of
reproducing kernels such that its biorthogonal system is not complete.
\end{theorem}

A converse result says that if $b_n$ have no more than a power decay,
then the biorthogonal systems are (almost) complete.

\begin{theorem}
\label{main2}
If there exists $N>0$ such that
\[
  \inf_m b_m(1+|t_m|)^N >0,
\]
then the orthogonal complement
to a system biorthogonal to an exact system of reproducing kernels
is finite-dimensional.

If, moreover, $\sum_n b_n=+\infty$, then
any system biorthogonal to an exact system of reproducing kernels is
complete in $\Hg$.
\end{theorem}

The restriction on the decay of $b_n$ in Theorem \ref{main2}
is essential.

\begin{example}
{\rm There exists a space $\Hg \in\mathfrak{R}$ such that
$\sum_n b_n=+\infty$, but there
exists an exact system of reproducing kernels such that its
biorthogonal is not complete.}
\label{b_n_infty_example}
\end{example}

Young's result about the Paley--Wiener space corresponds to the
situation when $t_n=n$, $n\in\mathbb{Z}$, and $b_n=1$, and follows
from Theorem \ref{main2}. E. Fricain have proved the completeness of
biorthogonal system in de Branges spaces under the assumption that
$\sup_{x\in\mathbb{R}}\varphi'(x)<+\infty$, $\varphi$ being the
phase function for $E$, that is, a smooth branch of the argument of
$E$ on $\mathbb{R}$. We show now that this assumption  implies a
lower estimate on $b_n$. We will use the de Branges decomposition
$E=A - iB$, where $A$ and $B$ are entire functions real on the real
axis. As known, reproducing kernels in the de Branges space
$\mathcal{H}(E)$ are given by
\begin{equation}
\label{he-repr} k_w(z)=\frac{1}{\pi}\cdot\frac{B(z)\cdot
A(\overline{w})- A(z)B(\overline{w})}{z-\overline{w}}.
\end{equation}
Without loss of generality we may assume that
$\{k_{t_n}/{\|k_{t_n}\|}\}$ is an orthonormal basis, where $t_n \in
\mathbb{R}$ are all solutions of the equation $A(z)=0$ \cite[Theorem
22]{DB}. Of course, $A$ is the corresponding generating function. In
this notation
\begin{equation}
\label{he-bn}
b_n= \frac{\|k_{t_n}\|^2}{|A'(t_n)|^2}=\frac{k_{t_n}(t_n)}{|A'(t_n)|^2}=
\frac{B(t_n)}{\pi A'(t_n)} = \frac{1}{\pi \varphi'(t_n)}.
\end{equation}
Hence, $\inf_n b_n>0$ and the result follows from Theorem \ref{main2}.

\subsection{Size of the orthogonal complement}
Now we turn to the question of "size"\, of the orthogonal
complement to a biorthogonal system in the case when it is
not complete. A precise definition of the "size" and the main
results are given in $\S3$. Here we only emphasize the following
informal principle:
\smallskip

{\it The size of the orthogonal complement to a
biorthogonal system depends on smallness of the sequence $\{b_n\}$.
The orthogonal complement becomes bigger if $b_n$ tend to zero
faster. If, however, $\{b_n\}$ are extremely small, then the
orthogonal complement is finite-dimensional. }
\smallskip

Now we give a precise formulation of the latter property. Put
$$
\ell^2(T, b) :=\big\{f:T\to\mathbb{C}:\ \sum_n |f(t_n)|^2 b_n<+\infty \big\}.
$$
The following result relates the size of the biorthogonal system to
the density of polynomials on discrete subsets of the real line.
Assume that $b_n$ are so small that
the polynomials belong to $\ell^2 (T,b)$ and are dense there, that is,
there is no non-trivial sequence $\{c_n\}\in \ell^2(T,b)$,
such that $\sum_n c_n n^k=0$ for any
$k\in\mathbb{N}_0$.

\begin{theorem}
Let $T\subset \mathbb{R}$ and assume that
the polynomials are dense in $\ell^2 (T,b)$.
If $\Hg$ is the Hilbert space of
the class $\mathfrak{R}$ corresponding to $\Hg(T,b)$,
then the closed linear span of the system biorthogonal to an exact
system of reproducing kernels always has finite codimension.
\label{moment}
\end{theorem}

Density of polynomials in the spaces of the form $\ell^2(T,b)$
is closely connected to the quasianalyticity phenomena
(see \cite{ko,bs}). We illustrate this by the following example
(further examples are given in Section 3).

\begin{example}
{\rm Let $t_n=n$, $n\in \mathbb{Z}$, and let $b_n=\exp(-|n|)$.
Then any biorthogonal system has finite codimension.
More generally, let $w=\exp(-\Omega)$ be an even
function such that $\Omega(e^t)$ is a convex function of $t$
($w$ is a so-called normal majorant).
Let $b_n = w(n)$. If $\sum_{n} \frac{|\log b_n|}{n^2+1} =+\infty$,
then any biorthogonal system has finite codimension. }
\end{example}

Finally, let us summarize:

$*$ If $b_n$ has at most power decay, then any biorthogonal system has
finite codimension (Theorem \ref{main2});

$*$ If the polynomials belong to $\ell^2(T,b)$ and are not dense there
("non-quasianalytic case"), then the codimension may be infinite
(see Proposition \ref{big_size});

$*$ If the polynomials are dense in $\ell^2(T,b)$
("quasianalytic case"), then the codimension is finite
(Theorem \ref{moment}).
\medskip

\subsection{Model subspaces}
Our next result is about general {\it model} or {\it star-invariant}
subspaces $K_\Theta$ of the Hardy space $H^2$ in
the upper half-plane $\mathbb{C}^+$.
Let $\Theta$ be an inner function in $\mathbb{C}^+$,
that is, a bounded analytic function such that
$\lim_{y\to +0}|\Theta(x+iy)|=1$ for almost all
$x\in\mathbb{R}$. With each $\Theta$ we associate the subspace
$$
  K_\Theta=H^2\ominus\Theta H^2.
$$
These subspaces, as well as their vector-valued generalizations
play an outstanding role both in function theory and operator theory.
For their numerous applications we refer to \cite{nk, nk1, nk2}.
It is well known that if $\Theta$ has a meromorphic continuation
to the whole plane, then $\Theta = E^*/E$ for a function $E$
in the Hermite--Biehler class and the mapping $f\mapsto Ef$
is a unitary operator from $K_\Theta$ onto $\mathcal{H}(E)$,
which maps reproducing kernels onto reproducing kernels.

The reproducing kernels of the space $K_\Theta$ are of the form
\begin{equation}
\label{repr-kern}
k_\lambda(z) =\frac{i}{2\pi} \cdot \frac{1-\overline{\Theta(\lambda)}\Theta(z)}
{z-\overline \lambda},\qquad \lambda\in \mathbb{C}^+.
\end{equation}
Reproducing kernels of the model spaces have a rich and subtle
structure and their geometric properties (such as completeness,
Bessel sequences, Riesz basic sequences) are still not completely
understood (see, e.g., \cite{hnp, mp, bar06} and \cite[Part D,
Chapter 4]{nk2}). In particular, it is an open problem whether any
model subspace has a Riesz basis of reproducing kernels. A special
case of the completeness problem for reproducing kernels is the
completeness of exponential systems in $L^2(-a,a)$ (corresponding to
$\Theta(z) = e^{2iaz}$) settled by the classical Beurling--Malliavin
theory. A recent breakthrough in the completeness problem for
reproducing kernels in model subspaces is due to N. Makarov and A.
Poltoratski \cite{mp, mp1} who suggested a new approach based on
singular integrals and extended the Beurling--Malliavin theory to
some classes of model subspaces and de Branges spaces.

The problem whether the system biorthogonal to an exact system
of reproducing kernels is complete in $K_\Theta$ was posed by N.K. Nikolski;
it was studied by E. Fricain in \cite{Fr} where for the class
of inner functions with bounded derivatives
a positive answer was obtained. Theorem \ref{main1},
which applies to the case of meromorphic inner functions,
already shows that the answer in general is negative.
However, we are able to prove an analog of Theorem \ref{main1}
for general model spaces.

Let $\sigma(\Theta)$ be the spectrum of the inner function
$\Theta$, that is, the set of all $\zeta\in\mathbb{R}\cup{\infty}$
such that $\lim_{z\rightarrow\zeta}\inf|\Theta(z)|=0$. Note that
$\sigma(\Theta)$ is closed and $\Theta$ (and any $f\in K_\Theta$)
has analytic continuation across any interval of the set
$\mathbb{R}\setminus\sigma(\Theta)$. A point $\zeta \in \mathbb{R}$
is said to be a {\it Carath\'eodory point} for $\Theta$
if $\Theta$ has
an angular derivative at $\zeta$, that is, there exists
the nontangential limit $\Theta(\zeta)$ with $|\Theta(\zeta)|=1$,
as well as the nontangential limit  $\Theta'(\zeta) =
\lim_{z\to\zeta}\frac{\Theta(z) -\Theta(\zeta)}{z-\zeta}$.

\begin{theorem}
Let $\Theta$ be an inner function in $\mathbb{C}^+$ such
that there exists $\zeta \in\sigma(\Theta)$
which is a Carath\'eodory point for $\Theta$.
Then there exists an exact system of reproducing
kernels such that the biorthogonal system is not complete.
\label{th-model}
\end{theorem}

If $\sigma(\Theta)=\{\infty\}$ (i.e. $\Theta$ is meromorphic
in $\mathbb{C}$)
the existence of the "angular derivative at $\infty$"
(appropriately defined) is equivalent to the condition
$\sum_n b_n<+\infty$ for some orthogonal basis of reproducing kernels and
we arrive at Theorem \ref{main1} for de Branges spaces. We mention also that
a result analogous to Theorem \ref{th-model}
holds for the model spaces in the unit disc (see Theorem \ref{th-model1}).

In Section 5, Theorem \ref{compl3},
we obtain a condition sufficient for the completeness
of a  biorthogonal system in a general model
space $K_\Theta$
in terms of the generating function $G$.

Throughout this paper, the notation $U(z)\lesssim V(z)$ (or
equivalently $V(z)\gtrsim U(z)$) means that there is a constant $C>0$
such that $U(z)\leq CV(z)$ holds for all suitable $z$.
We write $U(z) \asymp V(z)$ if $U(z)\lesssim V(z)$
and $V(z)\lesssim U(z)$.


\section{Proof of Theorems \ref{main1} and \ref{main2}}

Let $\{k_{\lambda_n}\}$ be an exact
system in $\Hg$. Without loss of generality we can assume that
$\{\lambda_n\}\cap\{t_m\}=\emptyset$, since we always can move
slightly the points $\{t_m\}$ so that $\{k_{t_m}/\|k_{t_m}\|_\Hg\}$
remains a Riesz basis. Suppose that $F$ and $G$ are generating functions of
systems $\{k_{t_m}\}$ and $\{k_{\lambda_n}\}$ respectively. Then the
system $\frac{G(z)}{G'(\lambda_n)(z-\lambda_n)}$ is biorthogonal to
$\{k_{\lambda_n}\}$.

For any $h\in \Hg$ we have a representation with respect to the
Riesz basis $\{k_{t_m} / \|k_{t_m}\|_\Hg \}$:
\begin{equation}
\label{h_rep}
   h=\sum_m \overline a_m\frac{k_{t_m}}{\|k_{t_m}\|_\Hg},\qquad \{a_m\}
   \in\ell^2.
\end{equation}
The last series converges in the norm and pointwise.

The function $h$ is orthogonal to $\frac{G(z)}{z-\lambda_n}$ for all
$n$ if and only if for any $n$
\[
\Big \langle \frac{G}{z-\lambda_n}, h \Big\rangle =
  \sum_m\frac{a_m}{\|k_{t_m}\|}\cdot\frac{G(t_m)}{t_m-\lambda_n}=0.
\]
Consider the meromorphic function
$$
   L(z):=\sum_m\frac{a_m}{\|k_{t_m}\|_\Hg}\cdot\frac{G(t_m)}{t_m-z}.
$$
The series converges uniformly on compact subsets 
of $\mathbb{C}\setminus T$, since $\frac{G}{z-\lambda_n} 
\in \Hg$ and so $\big\{\frac{G(t_m)}{t_m \|k_{t_m}\|_\Hg}\big\} \in \ell^2$.
The function $LF$ is entire and vanishes at the points
$\{\lambda_n\}$. Hence, $S := LF\slash G$ is an entire function, and
\begin{equation}
    \frac{G(z)S(z)}{F(z)}=\sum_m\frac{a_m}{\|k_{t_m}\|_\Hg}
    \cdot\frac{G(t_m)}{t_m-z}.
\label{main_eq}
\end{equation}
It follows that $a_m=S(t_m)\frac{\|k_{t_m}\|_\Hg}{F'(t_m)}
=\frac{|F'(t_m)|}{F'(t_m)}S(t_m)b^{1/2}_m$.
Hence,
\begin{equation}
    \sum_m|S(t_m)|^2b_m<+\infty. \label{norm}
\end{equation}
We can consider functions $S$ from \eqref{main_eq} which satisfy
\eqref{norm} as parametrization of all functions $h$ orthogonal to a given
biorthogonal system $\{\frac{G(z)}{z-\lambda_n}\}$. We denote the
space of all such functions $S$ by $\mathcal{S}$. It is a Hilbert
space with respect to the norm given as the square root of
the left-hand side of \eqref{norm}. Moreover, the mapping
$$
S\mapsto \sum_m  \frac{|F'(t_m)|}{\overline{F'(t_m)}}\cdot\overline{S(t_m)}
\,b^{1/2}_m \cdot\frac{k_{t_m}}{\|k_{t_m}\|_\Hg}
$$
is a unitary operator from $\mathcal S$ onto the orthogonal complement
of the system $\Big\{\frac{G(z)}{z-\lambda_n}\Big\}$.
\medskip

Now we turn to the proof of Theorem \ref{main1}. From \eqref{param} we get
\begin{proposition}
Function $M$ is in $\Hg$ if and only if
\begin{equation}
\label{repr1}
  \frac{M(z)}{F(z)}=\sum_n\frac{c_n}{z-t_n},
  \qquad \sum_n\frac{|c_n|^2}{b_n}<+\infty.
\end{equation}
Here the series converges uniformly on compact sets in $\mathbb{C}\setminus T$,
while the series $\sum_n c_n \frac{F(z)}{z-t_n}$
converges in the norm of the space $\Hg$.
\label{prop}
\end{proposition}

\begin{proof} (of Theorem \ref{main1}).
We want to construct a generating function $G$ of an exact system of
reproducing kernels such that
$\frac{G(z)}{F(z)}=\sum_n\frac{d_n}{z-t_n}$, where the series 
converges uniformly on compact subsets of $\mathbb{C}\setminus T$. 
If such a function $G$ is
constructed, we can take $S\equiv 1$ in \eqref{main_eq} and the
function 
$$
h = \sum_n b_n^{1/2}\cdot \frac{|F'(t_n)|}{\overline{F'(t_n)}}
\cdot \frac{k_{t_n}}{\|k_{t_n}\|_{\Hg}}
$$ 
is orthogonal to all $\frac{G(z)}{z-\lambda_n}$.

We choose coefficients $c_n$ so that $c_nt_n>0$,
\begin{equation}
\label{1and2}
    (1)\,\ \sum_n \frac{|c_n|^2}{b_n}< +\infty; \qquad
    (2)\,\ \sum_n \frac{(c_n t_n)^2}{b_n}=+\infty,
\end{equation}
and put $G(z)=F(z)\sum_n\frac{c_nt_n}{z-t_n}$. It follows from
\eqref{1and2} that $G\notin\Hg$, but
$\frac{G(z)}{z-\lambda_n}\in\Hg$ where $\lambda_n$ are zeros of $G$.
Without loss of generality we can assume that $G$ has no multiple
roots because it is enough to change one coefficient $c_0$ a little.
Indeed, we can write
$$
G(z)=c_0 \frac{F(z)}{z-t_0}+H(z) =c_0 F_1(z)+H(z),
$$
and the functions $F_1$ and $H$ have no common zeros. So, $G$ and
$G'$ have common zeros only at the points $z$ where
$F_1'(z)H(z)-H'(z)F_1(z)=0$ and $c_0F_1(z)+H(z)=0$, but this is
possible only for a countable set of coefficients $c_0$.

To prove the completeness of $\{k_{\lambda_n}\}$ we need to show
that there is no entire function $T$ such that $TG\in\Hg$. To prove this, we
introduce some additional requirements on $c_n$. Choose  a
subsequence of indexes $\{n_k\}$ so rare that
$|t_{n_{k+1}}|>2|t_{n_k}|$ and the disks
$D^1_k=\{z:|z-t_{n_k}|\leq\frac{t_{n_k}}{10}\}$ are pairwise
disjoint. For other indices we choose a sequence of positive numbers
$h_n$ with $\sum_n h_n<1$ such that the disks $D^2_{n}=\{|z-t_n|\leq
h_n\}, n\notin\{n_k\}$, are also pairwise disjoint. Now assume that
additionally to (\ref{1and2}) we have
$$
    \sum_k  c_{n_k} t_{n_k} <+\infty.
$$
This will be achieved if we assume $\sum_k |t_{n_k}|^{-2}<+\infty$,
$\sum_k b_{n_k}^{1/2}<+\infty$ and put $c_{n_k} := (b_{n_k})^{1/2}
t_{n_k}^{-1}$.
Now we make $c_n$ for $n\not\in\{n_k\}$ so small that
$$
\sum_{n\notin \{n_k\}}  \frac{|c_n t_n^2|}{h_n}
<\frac{1}{10} \sum_k  c_{n_k} t_{n_k}, \qquad
\sum_{n\notin \{n_k\}}   c_n t_n
<\frac{1}{10} \sum_k  c_{n_k} t_{n_k}.
$$

Assume that $TG\in\Hg$. Since $TG/F$ should be represented as
in \eqref{repr1}, we have
\begin{equation}
\label{12}
   \frac{T(z)G(z)}{F(z)}=\sum_n\frac{c_nt_nT(t_n)}{z-t_n},
   \qquad  \sum_n \frac{|c_nt_nT(t_n)|^2}{b_n}<\infty.
\end{equation}
We can estimate $G(z)\slash F(z)$ for
$z\notin (\cup D^1_k) \cup (\cup D^2_n)$:
\[
\begin{aligned}
  \frac{|G(z)|}{|F(z)|}
& =\Big|\sum_n\frac{c_nt_n}{z-t_n}\Big|\geq
  \frac{1}{|z|} \sum_n c_n t_n
  -\sum_{n\notin\{n_k\}}\frac{|c_nt^2_n|}{|z(z-t_n)|}-
  \sum_{n\in\{n_k\}}\frac{|c_nt^2_n|}{|z(z-t_n)|} \\
& \ge \frac{1}{|z|}\sum_n c_nt_n-\sum_{n\notin\{n_k\}}\frac{|c_n t^2_n|}{|z|h_n}-
  \sum_{n\in\{n_k\},|t_n|<|z|\slash3}\frac{c_n t_n}{2|z|}-10\sum_{n\in\{n_k\},
  |t_n|\geq|z|\slash3}\frac{ c_n t_n}{|z|}.
\end{aligned}
\]
Note that the last sum is $o(|z|^{-1})$ as $|z|\to\infty$. Hence,
for sufficiently large $|z|$
$$
\frac{|G(z)|}{|F(z)|} \ge \frac{1}{4|z|} \sum_k  c_{n_k} t_{n_k}
\gtrsim\frac{1}{|z|}, \qquad z\notin (\cup D^1_k ) \cup (\cup D^2_n).
$$
Using analogous estimates we can show that
$$
\frac{|T(z)G(z)|}{|F(z)|}\lesssim 1, \qquad
z\notin (\cup D^1_k) \cup (\cup D^2_n).
$$
Hence, $|T(z)|\lesssim 1+|z|$ for $z\notin (\cup D^1_k) \cup (\cup
D^2_n)$. By the choice of $t_{n_k}$ and $h_n$ there exist circles
$\Gamma_n =\{z: |z|=r_n \}$ with $r_n \to \infty$, such that
$\Gamma_n\cap \big( (\cup D^1_k) \cup (\cup D^2_n) \big)
=\emptyset$. Therefore, $T(z)=az+b$. But this contradicts
\eqref{12} and (2) in \eqref{1and2} unless $T\equiv 0$.
\end{proof}

\begin{remark}
{\rm Note that, by the choice of the coeficients $d_n = c_n t_n > 0$, all zeros
of the function $G$ constructed in the proof of Theorem \ref{main1}
are real. }
\end{remark}

In the proof of Theorem \ref{main2}  we will use the following lemma.

\begin{lemma}
If $S\in\mathcal{S}$, then $\frac{S(z)-S(w)}{z-w}\in\mathcal{S}$ for
any $w\in \mathbb{C}$. In particular, if $\mathcal{S}$ is of finite
dimension $n+1$, then $\mathcal{S}$ coincides with the set
$\mathcal{P}_{n}$ of all polynomials of degree at most $n$.
\label{lem}
\end{lemma}

\begin{proof}
Let $\lambda_0$ be a zero of $G$. Then $\frac{G(z)}{z-\lambda_0} \in\Hg$ and
\begin{equation}
\label{main_eq1}
\frac{G(z)}{(z-\lambda_0)F(z)}=\sum_m\frac{G(t_m)}{(t_m-\lambda_0) F'(t_m)(z-t_m)}.
\end{equation}
We have the identity
\[
  \begin{aligned}
  \frac{G(z)(S(z)-S(w))}{F(z)(z-w)}
  & =\frac{1}{z-w} \bigg(\frac{G(z)S(z)}{F(z)}-\frac{G(w)S(w)}{F(w)}\bigg) \\
  & +\frac{(w-\lambda_0)  S(w)}{z-w}\bigg(\frac{G(w)}{(w-\lambda_0) F(w)}-
  \frac{G(z)}{(z-\lambda_0) F(z)}\bigg)  -S(w)\frac{G(z)}{(z-\lambda_0)F(z)}.
  \end{aligned}
\]
By \eqref{main_eq} and \eqref{main_eq1},
\[
\frac{1}{z-w} \bigg(\frac{G(z)S(z)}{F(z)}-\frac{G(w)S(w)}{F(w)}\bigg)
  =\sum_m\frac{a_m}{(t_m-w)\|k_{t_m}\|_{\Hg}}\cdot\frac{ G(t_m)}{t_m-z},
\]
\[
  \frac{1}{z-w}\bigg(\frac{G(z)}{(z-\lambda_0) F(z)}-\frac{G(w)}
  {(w-\lambda_0)F(w)}\bigg)
  =\sum_m\frac{a_m}{(t_m-\lambda_0) (t_m-w) F'(t_m)}
  \cdot\frac{G(t_m)}{t_m-z}.
\]
Thus, we have shown that the function $\frac{G(z)}{F(z)} \cdot
\frac{S(z)-S(w)}{z-w}$ can
be represented as a series in \eqref{main_eq}. Condition \eqref{norm} for
the function $\frac{S(z)-S(w)}{z-w}$ follows from \eqref{adm}.
\end{proof}

\begin{proof} (of Theorem \ref{main2})
Assume that the system $\{\frac{G(z)}{z-\lambda_n}\}$ is not
complete. Fix $h$ as in \eqref{h_rep} which is orthogonal to all
functions $\frac{G(z)}{z-\lambda_n}$ and consider the corresponding
space $\mathcal{S}$. If $\mathcal{S}$ is an
infinite dimensional space then we can find a function
$S\in\mathcal{S}$ with at least $N+2$ zeros $w_1,...,w_{N+1}$
different from the points $\{\lambda_n\}$ and $\{t_n\}$. By Lemma
\ref{lem}
$T(z):=\frac{S(z)}{\prod_{l=1}^{N+1}(z-w_l)}\in\mathcal{S}$.
Moreover the corresponding coefficients from \eqref{main_eq} for $T$
are equal to $\frac{a_m}{\prod_{l=1}^{N+1}(t_m-w_l)}$.

Recall that
$$
f_m(z) := \frac{\|k_{t_m}\|_{\Hg}}{F'(t_m)}\cdot \frac{F(z)}{z-t_m} =
b_m^{1/2} \cdot \frac{|F'(t_m)|}{\overline{F'(t_m)}} \cdot \frac{F(z)}{z-t_m}
$$
is the biorthogonal system to the Riesz basis
$\{k_{t_m}/\|k_{t_m}\|_{\Hg}\}$ and thus also is a Riesz basis.
We have
\[
\begin{aligned}
G(z)T(z) & =F(z)\sum_m\frac{a_m}{\prod_{l=1}^{N+1}(t_m-w_l)}\cdot
\frac{G(t_m)}{\|k_{t_m}\|_{\Hg}(z-t_m)}=\\
& =\sum_m\frac{a_m}{b^{1/2}_m\prod_{l=1}^{N}(t_m-w_l)}
\cdot\frac{G(t_m)}{(t_m-w_{N+1})\|k_{t_m}\|_{\Hg}}
\cdot b_m^{1/2} \frac{F(z)}{z-t_m} =:
\sum_m d_m f_m(z).
\end{aligned}
\]
Note that $\inf_m|b^{1/2}_m\prod_{l=1}^{N}(t_m-w_l)|>0$ and
if we fix a zero $\lambda_0$ of $G$,
then
\[
   \Big|\frac{G(t_m)}{t_m-\lambda_0} \Big|
   \leq \Big\| \frac{G(z)}{z-\lambda_0}\Big\|_{\Hg}
   \cdot\|k_{t_m}\|_{\Hg}.
\]
So the coefficients $\{d_m\}$ are in $\ell^2$ and
$GT\in\Hg$. However, this contradicts the completeness
of the system $\{k_{\lambda_n}\}$.

Finally, assume that $\sum_n b_n = +\infty$.
If $\mathcal{S}$ is a finite-dimensional space,
then it follows from Lemma \ref{lem} that $\mathcal{S}$
is the space of polynomials $\mathcal{P}_n$ for some $n$, which
can not be true since $\sum_n |S(t_n)|^2 b_n = +\infty$
for any $S\in\mathcal{S}$. Thus, the biorthogonal system is complete.
\end{proof}

At the end of the section we prove Example \ref{b_n_infty_example}.

\begin{proof}
We will construct a space of the form $\mathcal{H}(\mathbb{Z},b)$, that is,
$t_n=n$, $n\in\mathbb{Z}$. The corresponding generating
function is $F(z)=\sin(\pi z)$.
Put
$$
S(z)=\prod_{k=1}^{\infty}\bigg(1-\frac{z^2}{(2^{k}+1\slash2)^2}\bigg),
\qquad G(z)=\cos(\pi z)\slash S(z).
$$
Then
\begin{equation}
    \frac{G(z)}{\sin(\pi
    z)}=\sum_{n=-\infty}^{+\infty}\frac{S^{-1}(n)}{\pi(z-n)}.
\label{S}
\end{equation}
Indeed, the difference between the left-hand side and the right-hand
side in (\ref{S}) should be an entire function of exponential type.
Since $\lim_{|z|\to\infty}\frac{G(z)}{\sin \pi z}=0$ along any
non-horizontal ray in $\mathbb{C}^+$ or $\mathbb{C}^-$, this
difference is zero.

Put $b_n=|S(n)|^{-2}$ for $n\neq \pm 2^k$ and $b_n=1$ otherwise.
We consider the space $\Hg\in\mathfrak{R}$
corresponding to $\mathcal{H}(\mathbb{Z},b)$. First of all we
want to show that $G$ is the generating function of an exact
system in $\mathcal{H}$. From \eqref{S} we
conclude that $\frac{G(z)}{z-\lambda_0}\in \mathcal{H}$
for any zero $\lambda_0$ of $G$. We need
to show that there is no nonzero entire $T$ such that
$TG\in\mathcal{H}(\mathbb{Z},b)$. If it exists then
\begin{equation}
    \frac{G(z)T(z)}{\sin(\pi z)}=
    \sum_{n=-\infty}^{\infty}\frac{S^{-1}(n)T(n)}{\pi(z-n)}, \qquad
    \sum_{n}\frac{|S^{-1}(n)T(n)|^2}{b_n}<+\infty.
\label{S1}
\end{equation}
It follows from (\ref{S1}) that $T$ is of zero exponential type, and,
at the same time, $\sum_k |T(2k+1)|^2<+\infty$, which implies $T\equiv 0$
(see, e.g., \cite[Lecture 21]{levin}).

By construction we automatically get $\sum_n b_n=+\infty$ and now it
remains to show that there exists a nonzero $S_1\in\mathcal{S}$. Let
$$
S_1(z)=\prod_{k=2}^{\infty}\bigg(1-\frac{z^2}{(2^{k}+\delta_k)^2}\bigg)
$$
and choose $\delta_k\in(0,1)$ so small that
$\sum_{k=1}^\infty|S_1(2^k)|^2<+\infty$. By a straightforward estimation
we get $|S_1(n)|\lesssim |S(n)|$. So, $\sum_n|S_1(n)|^2 b_n<+\infty$.
On the other hand,
$\lim_{y\rightarrow\pm\infty}\frac{|S_1(iy)|}{|S(iy)|}=0$ and we
have representation $\eqref{main_eq}$ with condition $\eqref{norm}$
for $S_1$.

Since $S_1$ is not a polynomial, we see that $\mathcal{S}$
is infinitely-dimensional. Thus, in this case the orthogonal
complement to the biorthogonal system is infinitely-dimensional.
\end{proof}


\section{Size of the orthogonal complement to a biorthogonal system}

In this section we will study in more detail the space $\mathcal{S}$
whose elements parameterize functions orthogonal to the biorthogonal
system via formula \eqref{main_eq}. First of all note that for any
$k\in\mathbb{N}_0$ it is possible that $\mathcal{S}$ coincides with
the set $\mathcal{P}_k$ of polynomials of degree at most $k$. On the
other hand, as we have seen in the proof of Example
\ref{b_n_infty_example}, it is possible that $\mathcal{S}$ is an
infinite-dimensional space. We introduce therefore the following
notion of the size of the orthogonal complement to the biorthogonal
system.

\begin{definition}
Let $\{k_\lambda\}$ be an exact system of reproducing kernels in a
space $\mathcal{H} \in\mathfrak{R}$ with the generating function $G$,
and let $\mathcal{S}$ be the corresponding space parametrizing the
orthogonal complement. Let $M$ be a positive increasing function on
$\mathbb{R}_+$. We say that the orthogonal complement to the
biorthogonal system has size $M$ if there exists $S\in\mathcal{S}$
such that, for some $y_0>0$,
\[
\log|S(iy)| \ge M(|y|), \qquad |y|>y_0.
\]
\end{definition}

From now on we will consider only the situation
when $T=\{t_n\} \subset \mathbb{R}$.
As we have mentioned in Introduction, this case corresponds
to de Branges spaces. Thus we assume that $\mathcal{H} =
\mathcal{H}(E)$ and the generating function $F$
of the sequence $T$ is given by $F=\frac{E+E^*}{2}$.
Our first theorem shows that under some mild restrictions on $t_n$
any function in $\mathcal{S}$ is of zero exponential type, and so,
for any $\varepsilon>0$,
the orthogonal complement can not have the size $M(r) =\varepsilon r$.

\begin{theorem}
\label{exp-type} Let $t_n\in \mathbb{R}$.
If $G$ is the generating function of an exact system
of reproducing kernels and $\mathcal{S}$ is the corresponding space,
then any $S\in\mathcal{S}$ is of zero exponential type.
\end{theorem}

From this we have an immediate corollary.

\begin{corollary}
\label{linear} Let $t_n \in\mathbb{R}$. Then
for any $\varepsilon>0$ the orthogonal complement of a biorthogonal
system can not have the size $M(r) =\varepsilon r$.
\end{corollary}

In what follows we will use essentially the inner-outer factorization
of $H^2$ functions (see, e.g., \cite{ko, nk}).
Recall that a function $f$ is said to be in the {\it Smirnov class} if
$f =g/h$, where $g,h$ are bounded analytic functions in $\mathbb{C}^+$
(or functions in $H^p$) and $h$ is outer.

\begin{proof} (of Theorem \ref{exp-type}).
If $\lambda_0$ is a zero of $G$, then $\frac{G}{z-\lambda_0} \in
\mathcal{H}(E)$. Hence, $h:=\frac{G}{(z-\lambda_0)E} \in H^2$.
Let us show that $h$ has no singular inner factor of the form $e^{2iaz}$,
$a>0$ (note that $h$ has no other singular factors since it is analytic
on $\mathbb{R}$). Indeed, if $e^{-2iaz}h \in H^2$, then put
$$
H(z) = e^{-iaz} \frac{\sin az}{z} G(z).
$$
Then $H/E\in H^2 $ and also $H^*/E \in H^2$, so $H\in \mathcal{H}(E)$,
which contradicts the fact that $\{\lambda: G(\lambda)=0\}$
is a uniqueness set for $\mathcal{H}(E)$.

Consider the inner function $\Theta = E^*/E$. Then
$2F = E(1+\Theta)$ and we have
$$
    \frac{G(z)S(z)}{E(z)}=\sum_m\frac{a_m}{\|k_{t_m}\|_\Hg}
    \cdot\frac{1+\Theta(z)}{t_m-z}\,G(t_m)=:f.
$$
We show that the right-hand side function $f$ is in the Smirnov class in
$\mathbb{C}^+$. First of all note that
if $v_m\ge 0$ and $\{v_m\} \in \ell^1$, then
$$
\ima \sum_m \frac{v_m}{t_m-z} >0, \qquad z=x+iy \in \mathbb{C^+},
$$
and so this sum is in the Smirnov class.
The same is obviously true also for an arbitrary sequence
$\{v_m\} \in \ell^1$.

Since $\frac{G}{z-\lambda_0} \in \mathcal{H}(E)$
we have $\{\|k_{t_m}\|_\Hg^{-1} t_m^{-1} G(t_m)\} \in \ell^2$.
Hence,
$\{v_m\}\in\ell^1$ where
$$
v_m = \frac{a_m}{\|k_{t_m}\|_\Hg} \cdot\frac{G(t_m)}{t_m},
$$
and we have
$$
\frac{f(z)}{1+\Theta(z)}
= \sum_m v_m \frac{t_m}{t_m-z} = \sum_m v_m +z\sum_m \frac{v_m}{t_m -z}.
$$
Hence $f$ is in the Smirnov class.

Thus $S = f \Big(\frac{G}{E}\Big)^{-1}$
is a ratio of two functions of bounded type,
and so is a function of bounded type (zeros of Blaschke products cancel).
Moreover, since $S$ is analytic on $\mathbb{R}$
it is in the Smirnov class unless it has a factor of the form $e^{-2iaz}$
in its canonical inner-outer factorization. However, it can not happen,
since, as we have seen,  $G/E$ has no singular inner factor.

By completely identical arguments, $S$ is in the Smirnov class
in the lower half-plane $\mathbb C^-$.
Since $S$ is in the Smirnov class both in
$\mathbb C^+$ and in $\mathbb C^-$, it is of zero exponential type
by M.G. Krein's theorem (see, e.g., \cite[Chapter I, Section 6]{hj}).
\end{proof}

As we have seen in Corollary \ref{linear},
the linear growth of the function $M$ (which determines
the size of the orthogonal complement) is not possible.
The following proposition, which applies to the case
$\mathcal{H}(\mathbb{Z},b)$ (that is, $t_n=n$, $n\in \mathbb{Z}$),
provides a converse result: for any slower growth,
the size $M$ for the orthogonal complement
may be achieved for some choice of $b_n$.

\begin{proposition}
Let $M(r)\slash r$ be a decreasing function which tends to zero when
$r\rightarrow+\infty$. Then there exist a sequence $b_n$ and an exact
system $\{k_\lambda\}$ in the space of entire functions $\Hg$
$($corresponding to the space $\mathcal{H}(\mathbb{Z},b)$$)$ such that the
orthogonal complement to the biorthogonal system has size $M$.
\label{big_size}
\end{proposition}

\begin{proof} First of all, we "atomize" $ $ function $M$. There
exists an increasing integer-valued function $\mu$ with jumps at
some half-integer points such that $M-\mu\in L^{\infty}$. We will
assume that $\mu(0)=0$, $\mu(1)=1$. Let $S(z)=\prod_{t\in\supp
d\mu}\big(1-\frac{z^2}{t^2}\big)$. We want to estimate $|S(iy)|$ for
large $|y|$:
\[
\begin{aligned}
  \log|S(iy)| & =\int_0^\infty\log\bigg(1+\frac{y^2}{t^2}\bigg)d\mu(t)=2y^2
  \int_0^\infty\frac{\mu(t)dt}{t(y^2+t^2)} \\
  & \geq
  y\Big( \inf_{t\in[1\slash2,y]} \frac{\mu(t)}{t}\Big) \cdot
  \int_{1\slash2}^y  \frac{y}{y^2+t^2}dt\gtrsim \mu(y), \qquad |y|\to\infty.
  \end{aligned}
\]
Now we put $G(z)=\cos(\pi z)\slash S(z)$,
$b_n=S^{-1}(n)$. The function $G$ has exponential type
$\pi$ and we can verify that $G$ is the generating function of some
exact system using the same arguments as in the proof of Example
\ref{b_n_infty_example}. Indeed, if $TG \in \Hg$, then $T$ is of
zero exponential type and $\sum_n |T(n)|^2<+\infty$, so $T\equiv 0$.
It only remains to note that $ (z-1/2)^{-1} S(z) \in \mathcal{S}$.
\end{proof}

As we have seen in the proof of Proposition \ref{big_size}  the size
of orthogonal complement to a biorthogonal system corresponds to the
speed of decrease of the coefficients $b_n$. It becomes bigger when
$b_n$ decrease faster. Nevertheless from Theorem \ref{moment} we
see that for extremely small $b_n$ biorthogonal system has finite
codimension. This corresponds to polynomial size
$M(r)= n \log r$.

\begin{proof}(of Theorem \ref{moment}).
Suppose $G$ is the generating function of some exact
system, and let $\lambda_0$ be a zero of $G$. Hence,
\[
    \frac{G(z)}{(z-\lambda_0)F(z)} =\sum_n \frac{c_n}{z-t_n},
    \qquad \sum_n\frac{|c_n|^2}{b_n}<+\infty.
\]
In particular, $\{c_n/b_n\} \in \ell^2(T, b)$. Since the polynomials are
dense in $\ell^2(T,b)$, it follows that there exists $N \in\mathbb{N}\cup\{0\}$
such that $\langle \{c_n/b_n\}, t_n^N \rangle_{\ell^2(T,b)}
=  \sum_n c_n t_n^N\neq 0$.
We take the smallest $N$ with this property.
Let us estimate $G/F$ from below on the imaginary axis.
We have
$$
\sum_n \frac{c_n}{iy-t_n} =
\sum_{|t_n|\le |y|/2} \frac{c_n}{iy-t_n} +
\sum_{|t_n|>|y|/2} \frac{c_n}{iy-t_n}.
$$
Since $\sum_n |c_n|\cdot|t_n|^k<\infty$ for any $k$, we have
$$
\Big|\sum_{|t_n|>|y|/2} \frac{c_n}{iy-t_n}\Big| =O\Big(\frac{1}{y^{N+2}}\Big),
\qquad |y|\to +\infty.
$$
For $|t_n|\le |y|/2$, we have
$$
(iy -t_n)^{-1}=\sum\limits_{k\ge 0}\frac{t_n^k}{(iy)^{k+1}}=
\sum\limits_{k=0}^{N}\frac{t_n^k}{(iy)^{k+1}}+
r_n(y)\frac{t_n^{N+1}}{(iy)^{N+2}},
$$
where $|r_n(y)|\le 2$. Hence,
$$
\begin{aligned}
\sum_{|t_n|\le |y|/2} \frac{c_n}{iy-t_n}
& =
\sum\limits_{k=0}^{N} \frac{1}{(iy)^{k+1}} \sum_n c_n t_n^k
-
\sum\limits_{k=0}^{N}
\frac{1}{(iy)^{k+1}} \sum_{|t_n|> |y|/2}  c_n t_n^k \\
& + \frac{1}{(iy)^{N+2}}\sum_n c_n r_n(y) t_n^{N+1} =
\frac{1}{(iy)^{N+1}} \sum_n  c_n t_n^N + O\Big(\frac{1}{y^{N+2}}\Big).
\end{aligned}
$$
We conclude that $|G(iy)|/|F(iy)| \ge c |y|^{-N-1}$, $|y|\to +\infty$.

Now let $S\in\mathcal S$ and so, for some $\{a_m\} \in\ell^2$,
$$
    \frac{S(z)G(z)}{F(z)} = \sum_m\frac{a_m}{\|k_{t_m}\|_\Hg}
    \cdot\frac{G(t_m)}{t_m-z}.
$$
We have
$$
    \Big|\sum_m\frac{a_m}{\|k_{t_m}\|_\Hg}
    \cdot\frac{G(t_m)}{t_m-iy}\Big|  \le \sum_m |a_m|\frac{|G(t_m)|}
    {\|k_{t_m}\|_\Hg|t_m|}=:A < +\infty,
$$
since $\{\|k_{t_m}\|^{-1}_\Hg |t_m|^{-1} G(t_m) \} \in\ell^2$.
Hence,
$$
|S(iy)|\le A\frac{|F(iy)|}{|G(iy)|} \le C |y|^{N+1}, \qquad |y|\to\infty.
$$
By Theorem \ref{exp-type}, $S$ is of zero exponential type, and thus,
a polynomial of degree at most $N+1$.
Hence, $\mathcal{S}$ has finite dimension.
\end{proof}

Using the known results on density of polynomials we can give more
examples of the situation where all biorthogonal systems have
finite codimension. These examples deal with one-sided sequences
with power growth.

\begin{example}
{\rm  1. Let $t_n=n^{1/\beta}$, $n\in\mathbb{N}$,
and let $b_n= \exp(-A t_n^\alpha)$, where $A,\alpha>0$.
If $\beta\ge 1/2$, then the polynomials are dense
in $\ell^2(T, b)$ if and only if $\alpha \ge 1/2$.
If $\beta <1/2$,
then the polynomials are dense in the space
$\ell^2(T, b)$  for $\alpha>\beta$ and are not dense for
$\alpha<\beta$; if $\alpha = \beta$, then the polynomials
are dense if and only if $A\ge \pi \cot \pi \beta$ (see \cite{bs}
for details).
\smallskip

2. The situation changes when we add a sparse sequence
on the negative semiaxis.
Let
$$
   t_n=
  \begin{cases}
   n^{1/\beta}, & n>0, \ n\in\mathbb{Z}, \\
   -2^{|n|},    & n<0, \ n\in\mathbb{Z}.
  \end{cases}
$$
Let $b_n= \exp(- t_n^\alpha)$, $\alpha>0$.
Using the results of \cite{bbh} one can show that
for $1<\beta<3/2$, the polynomials are dense in $\ell^2(T, b)$
for $\alpha > 1/2$ and not dense
for $\alpha < 1/2$. For $3/2 \le \beta<2$,
the polynomials are dense in $\ell^2(T, b)$
for $\alpha > \beta-1$ and  not dense for $\alpha < \beta-1$. }
\end{example}

\section{Incomplete biorthogonal systems in model subspaces}
\label{model1}

We start with an analogue of Theorem \ref{th-model} for the unit
disc $\mathbb{D}$. Let $\Theta$ be an inner function in $\mathbb{D}$
and, as in the half-plane case,  let $K_\Theta = H^2\ominus\Theta
H^2$. Denote by $\sigma(\Theta)$ the boundary spectrum of $\Theta$.
A point $\zeta \in \mathbb{R}$ is said to be a {\it Carath\'eodory
point} if $\Theta$ has an angular derivative at $\zeta$, that is,
there exist the nontangential limit $\Theta(\zeta)$ with
$|\Theta(\zeta)|=1$, as well as the nontangential limit
$\Theta'(\zeta) = \lim_{z\to\zeta}\frac{\Theta(z)
-\Theta(\zeta)}{z-\zeta}$. By the Ahern--Clark theorem \cite{ac70},
this is equivalent to the fact that the reproducing kernel $k_\zeta$
belongs to $K_\Theta$ and any element in $K_\Theta$ has finite
nontangential boundary value at $\zeta$. Finally, if $\Theta = B
I_\nu$ is a factorization of $\Theta$ into a Blaschke product and a
singular inner function, then $\zeta$ is a Carath\'eodory point if
and only if
$$
|\Theta'(\zeta)|=\sum\limits_n \frac{1-|z_n|^2}{|\zeta-z_n |^2}+
\int\limits_\mathbb{T}\frac{d\nu(\tau)}{|\zeta - \tau|^2} <+\infty.
$$
Here $z_n$ are zeros of $B$ and $\nu$ is a singular measure on the
circle $\mathbb{T}$.

\begin{theorem}
\label{th-model1}
Let $\Theta$ be an inner function in $\mathbb{D}$
such that there exists $\zeta\in \sigma(\Theta)$
which is a Carath\'eodory point.
Then there exists an exact system of reproducing
kernels such that its biorthogonal system is not complete.
\end{theorem}

Without loss of generality we may assume that $\Theta$ is a Blaschke product.
Indeed, we can always pass to a Frostman shift
$B = \frac{\Theta-\gamma}{1-\overline \gamma \Theta }$,
$|\gamma|<1$, which is
a Blaschke product, and the map $f\mapsto (1-|\gamma|^2)^{1/2}
\frac{f}{1-\overline\gamma \Theta}$
is a unitary operator from $K_\Theta$ onto $K_B$, which
maps the kernels to the kernels (up to constant bounded factors).

We also will move to the upper half-plane so that $\zeta$ goes to
$\infty$. Recall that for a Blaschke product $B$, the
property to have an
"angular derivative at $\infty$" is equivalent to any of the
following:

a) there exists $\alpha\in \mathbb{C}$ with $|\alpha|=1$ such that
$$
\rea \frac{\alpha+\Theta(z)}{\alpha-\Theta(z)}=p\, \ima z
+\frac{\ima z}{\pi}
\int\limits_\mathbb{R}\frac{d \mu(t)}{|t-z|^2},
\qquad z\in\mathbb{C^+},
$$
for a singular measure $\mu$ and $p>0$;

b) there exists a unimodular constant $\alpha$ such that
$\alpha - B \in K_B$;

c) there exist a unimodular $\alpha$ and $q>0$ such that
$$
1-\overline \alpha B(iy) =
\frac{q}{y} +o\Big(\frac{1}{y}\Big), \qquad y\to +\infty.
$$
In this case
\begin{equation}
\label{q-form}
q=2p^{-1} = 2\sum_n y_n,
\end{equation}
where $z_n=x_n+iy_n$ are zeros of $B$ (and, in particular, the
series $\sum_n y_n$ converges). Of course we will assume $\alpha=1$,
so $1-B\in K_B$. Note also that for any $g\in K_B$ there exists
a finite limit $\lim_{y\to\infty} yg(iy)$, and
$$
(g,1-B) = 2\pi \lim_{y\to\infty} yg(iy).
$$

In what follows we again use essentially the inner-outer factorization
of $H^2$ functions. If $m\ge 0$
and $\log m\in L^2\Big(\frac{dt}{t^2+1}\Big)$ we denote by
$O_m$ the outer function with the modulus $m$ on $\mathbb{R}$.

If we identify the functions in $K_\Theta$ and their boundary values
on $\mathbb{R}$, then an equivalent definition of $K_\Theta$ is
$K_\Theta = H^2 \cap \Theta \overline{H^2}$. Thus, we have a criterion
for the inclusion $f\in K_\Theta$ which we will repeatedly use:
\begin{equation}
\label{crit}
f\in K_\Theta \ \Longleftrightarrow \ f\in H^2 \ \ \text{and} \ \
\Theta \overline f\in H^2.
\end{equation}

In the following lemmas we always assume that $\Theta$
is an inner function such that $\infty$ is a Carath\'eodory point and
$1-\Theta \in K_\Theta$.

Our first lemma shows that the zeros of $K_\Theta$ function may be
concentrated in the upper half-plane.

\begin{lemma}
\label{lem1}
Let $f=O_m B I \in K_\Theta$, where
$B$ is a Blaschke product and $I$ is some inner function.
Then there exists a function $g= O_{\tilde m} \tilde B \in K_\Theta$
such that $\Theta \overline g = O_{\tilde m}$ is outer,
$B$ divides the Blaschke product $\tilde B$,
and $|O_{\tilde m}| \asymp |O_m|$.

Moreover, if $\lim_{y\to\infty} y f(iy)=0$, we can choose $g$
so that $\lim_{y\to\infty} y g(iy)=0$.
\end{lemma}

\begin{proof}
By the criterion \eqref{crit}, we have $\Theta\overline f \in H^2$,
and hence, $\Theta \overline f = O_m J$ for some inner function $J$.
Then, again by \eqref{crit}, the function $f_1 = O_m B IJ $ is in
$K_\Theta$ and $\Theta \overline f_1 = O_m$. If $IJ$ is a Blaschke
product we are done. Assume now that $f_1 = O_m B_1 K$, where $K$ is
a singular inner function. Replace $K$ by its Frostman shift $K_1 =
\frac{K-\gamma}{1-\overline \gamma K}$, $|\gamma|<1$, which is a
Blaschke product. Put
$$
g = O_m (1-\overline \gamma K) K_1 B_1  = O_m (K-\gamma) B_1,
\qquad \ O_{\tilde m} = O_m (1-\overline \gamma K).
$$
Then $\Theta \overline g = \Theta \overline O_m (\overline K-\overline \gamma)
\overline B_1 =
O_m (1-\overline \gamma K)$ since $\Theta \overline O_m
\overline B_1 = K O_m$. Thus, $g \in K_\Theta$ and $|O_{\tilde m}|
\asymp |O_m|$.

Finally, note that if $\lim_{y\to\infty} y f(iy)=0$, that is,
$f$ is orthogonal to $1-\Theta$, then the same is true for $f_1$.
Hence,
$$
0=(f_1, 1-\Theta)= (\overline \Theta f_1, \overline \Theta -1)
= (\Theta-1, \Theta \overline f_1) = (\Theta-1, O_m).
$$
Thus, $\lim_{y\to\infty} y O_m(iy)=0$, and the same is true for $g$,
since $|O_{\tilde m}| \asymp |O_m|$.
\end{proof}

The next lemma shows that one can get rid of real zeros
of a function in $K_\Theta$. Here we will use a lemma
due to N. Makarov and A. Poltoratski \cite{mp}.

\begin{lemma}
\label{lem2}
Let $f$ be a function in $K_\Theta$, let $a_n \in \mathbb{R}$, $a_n<a_{n+1}$,
$|a_n|\to\infty$, $n\to\infty$, and assume that there exist
nonnegative integers $m_n$ such that
$$
(t-a_n)^{-m_n} f \in L^2(a_n-\delta_n, a_n+\delta_n)
$$
for a small $\delta_n>0$, but $(t-a_n)^{-m_n-1} f \notin
L^2(a_n-\delta_n, a_n+\delta_n).$
Let $\{w_j\}$ be a sequence of points in
$\mathbb{C}^+$ with $|w_j|\to\infty$.
Then there exists a function $h \in K_\Theta$
such that $f/h$ is locally bounded on each open interval
$(a_n, a_{n+1})$, for any $n$ we have $\frac{h}{z-a_n} \notin
L^2(a_n-\delta_n, a_n+\delta_n)$, and
$$
|h(w_j)|\gtrsim  |f(w_j)|.
$$

Moreover, if $\lim_{y\to\infty} yf(iy)=0$, we can take $h$
so that $\lim_{y\to\infty} y h(iy)=0$.
\end{lemma}

\begin{proof}
Assume for a moment that all $m_n=1$.
Then we divide $f$ by a function of the form $1-J$
where $J$ is a meromorphic Blaschke product and $J=1$ exactly at $\{a_n\}$.
We also want to do this so that $f/(1-J)$ is still in $H^2$
(and hence in $K_B$). Such a choice of $J$ is possible by
\cite[Lemma 3.15]{mp}. The function $J$ should be constructed as
$$
\frac{1+J(z)}{1-J(z)}  =  \frac{1}{i} \sum_n \frac{\nu_n}{a_n-z},
$$
where $\nu_n>0$ are very small (the decay depends on the norms of
$(t-a_n)^{-m_n} f$ in $L^2(a_n-\delta_n,a_n+\delta_n)$).
Since the right-hand side has a positive real part, $J$
is a meromorphic inner function, and also,
$$
\frac{2J(z)}{J(z)-1}  = 1+ \frac{1}{2i} \sum_n \frac{\nu_n}{a_n-z}.
$$
Then we have
$$
h= \frac{2J}{J-1}f \in H^2.
$$
Also, $\Theta\overline h = \Theta \overline f \cdot \frac{1}{1-J}
\in H^2$, and so $h\in K_\Theta$ by \eqref{crit}. Obviously, $f/h$
is locally bounded on each interval $(a_n, a_{n+1})$, Since
$\lim_{y\to\infty}\frac{2J(iy)}{1-J(iy)} =-1$, we have $\lim yh(iy)
=0$. Finally, note that by choosing $\nu_n$ sufficiently small we
can make $\frac{|2J(w_j)|}{|1-J(w_j)|}$ as close to 1 as we want.

In general case when $m_n\ne 1$ we repeat the procedure and
construct a sequence of meromorphic inner functions $J_k$ so that
$a_n$ is in the set $\{t:J_k(t)=1\}$ exactly for $m_n$ of the
functions $J_k$. Put
$$
h= g \prod_k \frac{2 J_k}{J_k-1}.
$$
Chosing the masses $\nu_n^k$ in the definition of $J_k$ sufficiently
small we may achieve that the product converges to $h\in K_\Theta$,
$\lim_{y\to\infty} yh(iy) =0$ and $|h (w_j) |\ge c |f (w_j) |$ for a
positive constant $c$. We omit the technicalities.
\end{proof}

\begin{lemma}
\label{lem3}
Let $f$ be a Smirnov class function in $\mathbb{C}^+$
such that $f$ restricted to $\mathbb{R}$ has
a Smirnov class extension
to $\mathbb{C}^-$. Assume also that $\Delta\subset \mathbb{R}$
is an open interval such that $f$ is in $L^2(\Delta)$.
Then $f$ has an analytic extension through $\Delta$.
\end{lemma}

\begin{proof}
Let $[x-r,x+r]$ be a subinterval of $\Delta$.
Since $f$ is in the Smirnov class and in $L^2(\Delta)$
we have $f(\cdot+i\varepsilon) \to f$ in $L^2([x-r,x+r])$
when $\varepsilon\to +0$, and also when $\varepsilon\to -0$.
Consider the contours $\gamma_+ = [x-r,x+r] \cup
\{z\in \mathbb{C}^+: |z-x|=r\}$
and $\gamma_- = [x-r,x+r] \cup
\{z\in \mathbb{C}^-: |z-x|=r\}$ with the standard orientations.
If $z_0\in \mathbb{C^+}$, $|z_0-x|<r$, we have
$$
f(z) = \frac{1}{2\pi i} \int_{\gamma_+}\frac{S(z)}{z-z_0}dz,
\qquad
0 = \frac{1}{2\pi i} \int_{\gamma_-}\frac{S(z)}{z-z_0}dz
$$
(step a bit from $\mathbb{R}$, apply the Cauchy formula
and then pass to the limit). Taking the sum we conclude that
$$
f(z) = \frac{1}{2\pi i} \int_{|z-x|=r}
\frac{S(z)}{z-z_0}dz, \qquad |z-x|<r, \ z\notin \mathbb{R},
$$
and, hence, $f$ has an analytic extension in the whole disc $|z-x|<r$.
\end{proof}

\begin{proof} (of Theorems \ref{th-model} and \ref{th-model1}).
As we have seen, it suffices to prove the theorem for the case
where $\Theta=B$ is a Blaschke product such that
$\infty\in  \sigma(B)$, $\infty$ is a Carath\'eodory point
and $1-B\in K_B$. The symbol $q$ has the same meaning as
in \eqref{q-form}.

The idea of the proof is to construct a function $g$ in $K_B$ such that
\smallskip

(i) $g=B_\Lambda O_m$ and $B \overline g= O_m$, $O_m$ is an outer function
in $K_B$;

(ii) $\Lambda'=\Lambda \cup \{\lambda_0\}$ is a uniqueness set for
$K_B$ for any $\lambda_0 \in \mathbb{C}^+ \setminus \Lambda$;

(iii) $g$ is orthogonal to $1-B$ which means that $\lim_{y\to\infty} y|g(iy)| =0$.
\smallskip

If such $g$ is constructed, then the system
$(z-\lambda_0) \frac{g(z)}{z-\lambda}$,
$\lambda\in \Lambda'$, is biorthogonal to a complete system
$\{k_\lambda\}_{\lambda\in \Lambda'}$, but it is not itself
complete, since it is orthogonal to $1-B$.

We will consider separately two cases with slightly different proofs.
Let us start with an easier case.
\smallskip
\\
{\bf Case 1.} Assume that
\begin{equation}
\label{easy}
\limsup_{y\to +\infty} y^2\bigg|1- B(iy) - \frac{q}{y}\bigg|=+\infty.
\end{equation}
Consider the function
$$
f(z) = 1-B(z) -\frac{iq}{z-\overline z_0},
$$
where $z_0=x_0+iy_0$ is some zero of $B$. Obviously,
$(f, 1-B) = \lim_{y\to\infty} y|f(iy)| = 0$. Also note that for
$t\in \mathbb{R}$,
$$
f(t) =  \frac{t-x_0+iy_0-iq}{t-x_0+iy_0} -B(t).
$$
Since $q>2 y_0$ we have for almost all $t$
\begin{equation}
\label{1}
|f(t)| \ge  \bigg| \frac{t-x_0+iy_0-iq}{t-x_0+iy_0} \bigg| -1 \ge \frac{C}{t^2}.
\end{equation}
Now applying Lemma \ref{lem1} to $f$ we obtain a function $g=O_m B_\Lambda
\in K_B$
such that $|g| \asymp |f|$ on $\mathbb{R}$
and $B\overline{g} =O_m$.

Let us show that
$\Lambda' = \Lambda \cup \{\lambda_0\}$ is a uniqueness set for
$K_B$ for any $\lambda_0 \in \mathbb{C}^+ \setminus \Lambda$.
Assume the converse. Then there exists a function $F\in K_B$
which vanishes on $\Lambda'$. We can write $F=Sg$
for a function $S$ which is analytic in $\mathbb{C^+}$. Then $S= (F/B_\Lambda)/O_m$
is in the Smirnov class in $\mathbb{C^+}$. Also, since $F\in K_B$
we have
$$
B \overline F = B \overline g \overline S = O_m\overline S \in K_B.
$$
So $S^*(z) = \overline {S(\overline z)} $ has a Smirnov class extension  to
the upper half-plane, or $S$ itself is a Smirnov class function in
$\mathbb C^-$. Also, in view of (\ref{1}),
$S$ is locally in $L^2$ on $\mathbb R$ and hence, by Lemma \ref{lem3},
$S$ is entire.  Since $S$ is in the Smirnov class both in
$\mathbb C^+$ and in $\mathbb C^-$, it is of zero exponential type
by M.G. Krein's theorem
(see, e.g., \cite[Chapter I, Section 6]{hj}).

On the other hand, applying \eqref{1} once again we conclude that
$S \in (t+i)^2 L^2(\mathbb{R})$, and so $S $ is a polynomial
of degree at most 1. On the other hand, we have
$S^* O_m \in K_B$. Note that
$$
|O_m(iy)| \ge |f(iy)| = \bigg| 1-B(iy) -\frac{q}{y} +
O\Big(\frac{1}{y^2}\Big)\bigg|.
$$
If $S$ is not a constant, then, by (\ref{easy}),
$\limsup |yS^*(iy)O_m(iy)| =+\infty$,
but this is impossible, since this limit is finite for any function in $K_B$.
\medskip

{\bf Case 2.} Now assume that (\ref{easy}) is not satisfied, that is, there exists
$M>0$ such that
\begin{equation}
\label{diff}
y^2\bigg|1- B(iy) - \frac{q}{y}\bigg| \le M
\end{equation}
for sufficiently large $y$. The proof for Case 1 does not work, since
if (\ref{diff}) is satisfied we can not claim that
$S^* O_m$ is not in $K_B$ when $S$ is a polynomial of degree 1;
we need to use the fact that $\infty$ is a point of the spectrum
(i.e., a limit point for the zeros of $B$), since
for finite Blaschke products the argument should fail.
However, we will again construct $g$ satisfying (i)--(iii).

{\bf Step 1.} Take a very sparse subsequence $z_n=x_n+iy_n$
of zeros of $B$ so that $|z_n| \to \infty$ and
$\{z_n\}$ is a Carleson interpolating sequence.
Thus, the functions $\sqrt{y_n} (z-\overline z_n)^{-1}$
form a Riesz sequence in $K_B$ (see, e.g., \cite[Lecture VIII]{nk}).
Also we may assume that $y_n<1$, $x_n >2M/q$ for all $n$, and
$|x_n-x_k| >x_n/2$ for any $n, k$, $n\ne k$.

Now take a sequence $(c_n)\in \ell^2$, $c_n>0$, such that
\smallskip

a) $\sum_n c_n^2 x_n^2 =+\infty$;
\smallskip

b) $\sum_n \sqrt{y_n} c_n =1$ (recall that $\sum_n y_n < +\infty$);
\smallskip

c) $q c_n >4 \sqrt{y_n}$.
\smallskip

It is easy to see that these conditions may be achieved, if $y_n$
tends to zero sufficiently fast (e.g., if $\sum_n \sqrt{y_n} = m<1$
take $c_n=1/m$). It follows from b) and the fact that $x_n>2M/q$, that
\smallskip

d) $q\sum \sqrt{y_n} c_n x_n >M$ (may be $+\infty$).
\smallskip

Put
$$
f(z) = 1-B(z) -iq\sum \frac{ \sqrt{y_n} c_n }{z-\overline z_n}.
$$
Obviously, this is a function in $K_B$ which is orthogonal to $1-B$.
Also, for sufficiently large $y$,
\[
\begin{aligned}
y^2|\ima f(iy)|
&= y^2\bigg| \ima \big(1-B(iy)\big) - q\sum \frac{\sqrt{y_n} c_n x_n}
{x_n^2 +(y+y_n)^2}\bigg|  \\
& \ge y^2 q\sum \frac{\sqrt{y_n} c_n x_n} {x_n^2 +(y+y_n)^2} -M \ge C>0
\end{aligned}
\]
by condition d) and the fact that $y^2|\ima (1-B(iy))| \le M$.

We will see that $f(z_n) = -\frac{q c_n}{\sqrt{y_n}} +O(1)$
(since $x_n$ tends to infinity very rapidly),
so $f$ has a good growth along $z_n$
and we will see later that $zf$ can not be in $H^2$.
\smallskip

{\bf Step 2.} Applying Lemma \ref{lem1} we would obtain from $f$
a new function $g\in K_B$ such that $B \overline g$ is outer.
However, $f$ may have real zeros and we shall divide them out first.
Consider the function
$$
R(t) = \Big|1-iq\sum \frac{\sqrt{y_n} c_n}{t-\overline z_n}\Big|^2 -1 =
\Big(1-iq\sum \frac{\sqrt{y_n} c_n}{t-\overline z_n}\Big)
\Big(1+iq\sum \frac{\sqrt{y_n} c_n}{t-z_n}\Big) -1.
$$
It is analytic on $\mathbb R$ which implies that
zeros $a_n$ of $R$ are of finite multiplicities and $a_n\to\infty$.
Hence, we can represent $\mathbb R = \cup [a_n, a_{n+1}]$ where
$a_n<a_{n+1} \to \infty$ and $R$ is locally separated from zero
on the open intervals $(a_n, a_{n+1})$. Then
$$
|f(t)|\ge \bigg|\Big|1-iq\sum \frac{\sqrt{y_n} c_n}{t-\overline z_n}\Big| -1\bigg|
$$
is locally separated from 0 almost everywhere on each of the
intervals $(a_n, a_{n+1})$.

By Lemma \ref{lem2} there exists a function $h \in K_B$
such that $\frac{h}{z-a_n} \notin L^2(a_n-\delta, a_n+\delta)$
for any $\delta>0$, $|h/f|$ is locally separated from $0$
on any interval $(a_n, a_{n+1})$ and
$|h(z_n)| \gtrsim |f(z_n)|$. Now we apply Lemma \ref{lem3}
to the function $h$ and obtain a function $g = O_m B_\Lambda
\in K_B$ such that $B\overline g = O_m$ is outer,
$g$ is separated from zero locally on the intervals
$(a_n, a_{n+1})$,
and $\lim_{y\to\infty} y g(iy) = 0$. Also we have
$$
|O_m(z_n)| \asymp |O_{|h|}(z_n)| \ge |h(z_n)| \gtrsim |f(z_n)|.
$$

{\bf Step 3.} To complete the proof we need to show that
$\Lambda \cup \{\lambda_0\}$ is a uniqueness set for $K_B$.
Assume the converse. Then $F= Sg \in K_B$ for some function
$S$ analytic in $\mathbb{C}^+$ which vanishes at $\lambda_0$.
As in the proof of Case 1,
$S$ is in the Smirnov class both in $\mathbb{C}^+$ and in $\mathbb{C}^-$.
Also $S = F/g$ is locally in $L^2$ on $(a_n, a_{n+1})$.
By Lemma \ref{lem3} $S$ is meromorphic with possible poles in $a_n$.
Since $R$ has zeros of finite multiplicities,
there exist $m_n$ such that near $a_n$ we have
$$
|f(t)|\ge C|R(t)|\ge C |t-a_n|^{m_n}
$$
for some $C>0$, and an analogous estimate holds for $g$. Hence, $S$
may have only poles at the points $a_n$. However, $(t-a_n)^{-1}g
\notin L^2(a_n-\delta,a_n+\delta)$, and we conclude that $S$ is an
entire function which is of zero exponential type by Krein's
theorem.

We have $|O_m (iy)| \gtrsim |f(iy)| \ge Cy^{-2}$, $y\to+\infty$.
So,
$$
|S(iy)| \le \Big|\frac{F(iy)}{B_\Lambda(iy)}\Big|
|O_{m}(iy)|^{-1} \le Cy^{3/2}
$$
and
$$
|S^*(iy)| \le \frac{|(B\overline F)(iy)|}{|O_{m}(iy)|} \le Cy^{3/2},
\qquad y\to +\infty.
$$
We conclude that $S$ is a polynomial of degree at most 1.
\smallskip

{\bf Step 4.}
To finish the proof of completeness of
$\{k_\lambda\}_{\lambda\in \Lambda\cup \{\lambda_0\}}$,
we need to show that $S$ is a constant and, thus, zero.
Note that we have $S^* O_{m} \in K_B$.
Let $S$ be a polynomial of degree 1. Then we have $z O_{m} \in H^2$,
and so, since $z_n$ is a Carleson sequence, we have
\begin{equation}
\label{carl}
\sum_n y_n |z_n|^2 |f(z_n)|^2
\le C \sum_n y_n |z_n|^2 |O_m(z_n)|^2  <+\infty.
\end{equation}
On the other hand,
$$
\rea f(z_n) = 1-\frac{q c_n}{2\sqrt{y_n}} -q\sum_{k\ne n}
\frac{(y_n+y_k)\sqrt{y_k} c_k}
{|z_n-\overline z_k|^2} =1- \frac{q c_n}{2\sqrt{y_n}} -d_n.
$$
Using the properties of $z_n$ we get
$$
|d_n| = q\bigg|\sum_{k\ne n} \frac{\sqrt{y_k} c_k(y_n+y_k)}
{|z_n-\overline z_k|^2}\bigg| \le \frac{4q}{x_n} \sum_{k\ne n} \sqrt{y_k} c_k,
$$
and hence $\sum_n y_n |z_n|^2|d_n|^2 <+\infty$.
Recall also that by c)
we have $\frac{qc_n}{2\sqrt{y_n}} >2$. Thus
$$
|z_n|\cdot |\rea f(z_n)| +|z_n d_n| \ge \frac{qc_n |z_n|}{4\sqrt{y_n}},
$$
but, by the choice of $c_n$,
$$
\sum_n |z_n|^2 |c_n|^2 =+\infty,
$$
so $\sum_n y_n|z_n|^2 |\rea f(z_n)|^2=+\infty$, a contradiction
to \eqref{carl}. Thus $S \equiv const$ and,
since $S(\lambda_0)=0$, we finally conclude
that $S\equiv 0$.
\end{proof}


\section{Sufficient conditions for the completeness}
\label{model2}

In this section we discuss conditions sufficient for the completeness
of the system biorthogonal to an exact system of reproducing kernels.
Since the model spaces generated by meromorphic inner functions
essentially coincide with de Branges spaces we may obtain
completeness for a class of model subspaces
as a corollary of Theorem \ref{main2}. However, it seems
that this method does not work for the general model spaces.

On the other hand, even if we can not say that {\it any}
system biorthogonal to an exact system of reproducing kernels
is complete, one may look for criteria for completeness of
the biorthogonal system in terms of the generating function.
Results of this type may have applications in the spectral theory,
since systems biorthogonal to systems of reproducing kernels
appear, e.g., as eigenfunctions of certain rank one perturbations
of unbounded selfadjoint operators \cite{by}.

The crucial property of the model spaces which makes it possible
to apply the ideas from the proof of Theorem \ref{main2}
is the representation of the functions in $K_\Theta$
via the so-called Clark measures \cite{cl}, which is
similar to \eqref{param}.
For any $\alpha\in\mathbb{C}$, $|\alpha|=1$,
the function $(\alpha+\Theta)
/(\alpha-\Theta)$ has positive real part in the upper half-plane.
Hence, there exist $p_\alpha\ge 0$ and a measure
$\mu_\alpha$ with $\int (1+t^2)^{-1} d\mu_\alpha(t)<+\infty$
such that
$$
  \rea \frac{\alpha+\Theta(z)}{\alpha-\Theta(z)}=p_\alpha \ima z
  +\frac{\ima z}{\pi}
  \int\limits_\mathbb{R}\frac{d \sigma_\alpha(t)}{|t-z|^2},
  \qquad z\in\mathbb{C^+}.
$$
The Clark theorem states that if $\mu_\alpha$
is purely atomic (that is, $\mu_\alpha=\sum\limits_n c_n \delta_{t_n}$,
where $\delta_x$ denotes the Dirac measure at the point $x$)
and $p_\alpha=0$, then
the system
$\{k_{t_n}\}$ of reproducing kernels
is an orthogonal basis in ${K_{\Theta}}$.
In the general case, if $p_\alpha=0$, then the mapping
\begin{equation}
\label{clark}
   (C_\mu g)(z) = (\alpha-\Theta(z)) \int \frac{g(t)}{t-z}d\mu_\alpha(t)
\end{equation}
is a unitary map from $L^2(\mu)$ onto $K_\Theta$. By a theorem of A.
Poltoratski \cite{polt}, any function $f\in K_\Theta$ has
nontangential boundary values $\mu_\alpha$-a.e. for any $\alpha$.

Note also that the following are equivalent:

a) $p_\alpha >0$ for some $\alpha$;

b) $\alpha -\Theta \in H^2(\mathbb{R})$
(i.e., $\infty$ is a Carath\'eodory point for $\Theta$).

c) $\mu_\beta (\mathbb{R})< +\infty $ for some $\beta\in \mathbb{C}$,
$|\beta|=1$.

To see the equivalence of a) and c) note that
if $\mu_\beta (\mathbb{R})< +\infty$ and $p_\beta=0$, then
$p_{-\beta} >0$.
\smallskip

For the model spaces generated by meromorphic inner functions
we have the following immediate corollary of Theorem \ref{main2}
which applies to the so-called {\it tempered inner functions},
that is inner functions such that $|(\arg \Theta)'(t)| \lesssim
|t|^N$ for some $N>0$.

\begin{theorem}
\label{compl1}
Let $\Theta$ be a meromorphic inner function.
Write $\Theta = \exp(2i\phi)$ on $\mathbb{R}$,
where $\phi$ is a smooth increasing function on $\mathbb{R}$.
Let $\{t_n\} = \{t\in \mathbb{R}: \Theta(t) = -1\}$.
If $\alpha - \Theta \notin L^2(\mathbb{R})$ for
any $\alpha \in \mathbb{C}$, $|\alpha|=1$,
and for some $N>0$ and $C>0$,
$$
\phi'(t_n) \le C (|t_n|+1)^N,
$$
then any system biorthogonal to an exact system of reproducing kernels
is complete in $K_\Theta$.
\end{theorem}

\begin{proof}
We may write $\Theta=E^*/E$ for a Hermite--Biehler function $E$.
Then the mapping $f \mapsto Ef$ is a unitary map of $K_\Theta$
onto $\mathcal{H}(E)$. The function
$A = \frac{E+E^*}{2} = \frac{E(1+\Theta)}{2}$
is the generating function for $t_n$,
and the functions $k_{t_n}(z) = \frac{E(t_n)}{\pi i}\cdot \frac{A(z)}{z-t_n}$
form an orthogonal basis of reproducing kernels in $\mathcal{H}(E)$
(see formula \eqref{he-repr}). By \eqref{he-bn},
$b_n = (\pi \phi'(t_n))^{-1}$. Also note that
$$
\sum_n b_n \delta_{t_n} =
\sum_n \frac{B(t_n)}{\pi A'(t_n)} \delta_{t_n}
$$
is the Clark measure $\mu_{-1}$ for $\Theta$.
Since $\alpha - \Theta \notin L^2(\mathbb{R})$ for
any $\alpha$ we conclude that $\sum_n b_n =+\infty$,
and the result follows from Theorem \ref{main2}.
\end{proof}

\begin{remark}
{\rm Theorem \ref{compl1} may be extended to the case when
the spectrum of $\Theta$ is a finite set
and $\phi'$ has at most power growth near each of
these points, that is, $\phi'(t) \lesssim |t-a|^{-N}$ near
$a\in \sigma(\Theta)$. However, it is not clear how to deal with
the case when $\sigma(\Theta)$ has nonempty interior or
when Clark measures have singular continuous parts. }
\end{remark}

The following theorem applies to general inner functions.
We now impose the restrictions on the generating function
$G$ in place of $\Theta$. Recall that $G\in (t+i) H^2$.
The following result shows that if $G$ is sufficiently regular,
that is, has at most power growth, then the biorthogonal
system is complete.

\begin{theorem}
\label{compl3}
Let $\Theta$ be an inner function
such that $1 - \Theta \notin L^2(\mathbb{R})$
and $\mu_{1}(\mathbb{R})=+\infty$ $($these conditions are fulfilled, e.g.,
if $\infty$ is not a Carath\'eodory point for $\Theta$$)$.
Let $G$ be the generating function of an exact system
of reproducing kernels $\{k_{\lambda_n}\}$,
$\lambda_n\in\mathbb{C}^+$. If for some $N>0$ and $C>0$
$$
|G(t)| \le C |t+i|^N \qquad \mu_1{-\rm a.e.},
$$
then the system $\frac{G(z)}{z-\lambda_n}$ is complete.
\end{theorem}

\begin{lemma}
\label{lem5}
Let $G$ be the generating function of an exact system
of reproducing kernels $\{k_{\lambda_n}\}$, $\lambda_n\in\mathbb{C}^+$.
Then $\Theta \overline G$ is an outer function
and $\frac{G}{z-x} \notin L^2(\mathbb{R})$
for any $x\in \mathbb{R}$.
\end{lemma}

\begin{proof}
Fix a zero $\lambda_0$ of $G$. Then we can write
$G=(z-\lambda_0) g$ where $g\in K_\Theta$ and $g$ vanishes
on $\{\lambda_n\}_{n\ne 0}$.
Then $\Theta \overline G = (t-\overline \lambda_0)
\Theta \overline g$ on $\mathbb{R}$. Since $g\in K_\Theta$
we have $\Theta \overline g \in K_\Theta$ and
we may write $\Theta \overline g = I O_m$.
As in the proof of Lemma \ref{lem1}, if $I$ is not a Blaschke product
we may replace it by its Frostman shift
$I_1 = \frac{I-\gamma}{1-\overline \gamma I}$, $|\gamma|<1$, which is
a Blaschke product. Then
$h = g (1-\overline \gamma I) I_1$ is in $K_\Theta$.
If $I \ne 1$, then $I_1$ is a nontrivial Blaschke product, and
so $h$ vanishes on
$\{\lambda_n\}_{n\ne 0}$ and on the zero set of $I_1$.
This contradicts the completeness of $\{k_{\lambda_n}\}$.

If for some $x\in \mathbb{R}$, $\frac{G}{z-x} \in L^2(\mathbb{R})$,
then by \eqref{crit} $\frac{G}{z-x}$ is in $K_\Theta$
and vanishes on $\{\lambda_n\}$.
\end{proof}

\begin{proof}  (of Theorem \ref{compl3}).
Let $\mu=\mu_1$. Since $1 - \Theta \notin L^2(\mathbb{R})$, we have
representation  \eqref{clark} with $\alpha=1$
for the elements of $K_\Theta$.
Let $h\in K_\Theta$ be orthogonal to all functions
$\frac{G(z)}{z-\lambda_n}$. Then, for any $n$,
\begin{equation}
\label{orthog}
  \Big \langle \frac{G}{z-\lambda_n}, h \Big\rangle_{H^2} =
  \Big \langle \frac{G}{z-\lambda_n}, h \Big\rangle_{L^2(\mu)} =
  \int \frac{G(t)\overline{h(t)}}{t-\lambda_n} d\mu(t)=0.
\end{equation}
Consider the function
\begin{equation}
\label{l}
L(z) = (1-\Theta(z)) \int \frac{G(t)\overline{h(t)}}{t-z} d\mu(t);
\end{equation}
it is analytic in $\mathbb{C}^+$ and vanishes at $\lambda_n$.
Hence we may write $L=SG$ for a function $S$
analytic  in $\mathbb{C}^+$,
\begin{equation}
\label{param1}
S(z)G(z) = (1-\Theta(z)) \int \frac{G(t)\overline{h(t)}}{t-z} d\mu(t).
\end{equation}
We denote by $\mathcal S$ the linear space of all functions $S$
for which \eqref{param1} holds with some $h\in L^2(\mu)$.

Note that $S$ has nontangential boundary values $\mu$-a.e.
and $G(t)S(t) = G(t)\overline{h(t)}$ $\mu$-a.e. Indeed,
$G$ is locally bounded, and so $G\overline h \chi_{(-r,r)} \in L^2(\mu)$
for any $r>0$ (by $\chi_E$ we denote the characteristic
function of a set $E$). Write
\begin{equation}
\label{boundval}
L(z) = (1-\Theta(z)) \int_{(-r,r)} \frac{G(t)\overline{h(t)}}  {t-z} d\mu(t)
+(1-\Theta(z)) \int_{\mathbb{R}\setminus (-r,r)}
\frac{G(t) \overline{h(t)}} {t-z} d\mu(t).
\end{equation}
Then the first term is in $K_\Theta$ and has nontangential
boundary values $G(x)\overline{h(x)}$ for $|x|<r$ $\mu$-a.e.
The second term is analytic for $|z|<r$ and tends to zero for a fixed $z$
when $r\to\infty$. Hence, $S(t)G(t) = G(t) \overline{h(t)}$
$\mu$-a.e.

By the arguments similar to Lemma \ref{lem} it is easy to show
that for any $w\in \mathbb{C}^+$ and $S\in \mathcal S$
we have $\frac{S(z)-S(w)}{z-w} \in \mathcal S$
(just replace the sum by the integral with respect to $\mu$).

Assume first that $S$ is infinite dimensional. Then,
as in the proof of Theorem \ref{main2}, there exists
$S\in\mathcal{S}$ with at least $N$ zeros $w_1,...,w_N
\in \mathbb{C}^+$ different from the points $\{\lambda_n\}$, and
$T(z):=\frac{S(z)}{\prod_{l=1}^{N}(z-w_l)}\in\mathcal{S}$.
For some $h\in L^2(\mu)$ we have
\[
G(z)S(z)  = (1-\Theta(z)) \int \frac{G(t) \overline{h(t)} } {t-z} d\mu(t),
\]
and so,
\[
G(z)T(z)  = (1-\Theta(z)) \int \frac{G(t)}{\prod_{l=1}^{N}(t-w_l)}\cdot
\frac{\overline{h(t)}} {t-z} d\mu(t).
\]
Since $|G(t)| \lesssim |t+i|^N$ $\mu$-a.e., we conclude that
$ \frac{G(t)}{\prod_{l=1}^{N}(t-w_l)}\cdot \overline{h(t)} \in L^2(\mu)$.
Thus, for the function $GT$ we have a representation of the form
\eqref{clark}, and so $GT \in K_\Theta$. This contradicts the completeness of
$\{k_{\lambda_n}\}$.

Now assume that $\mathcal S$ is finite dimensional.
Since the transform $(\mathcal{D}_w S) (z) = \frac{S(z)-S(w)}{z-w} $
preserves the class $\mathcal S$, the functions
$S, \mathcal{D}_w S, \mathcal{D}^2_w S, \dots, \mathcal{D}^N_w S$
are linearly dependent for large $N$. We
conclude that $\mathcal S$ consists of rational functions.
Write $S=P/Q$, where $P,Q$ are polynomials without common zeros.
Since $S$ is analytic in $\mathbb{C}^+$, $Q$ has no zeros in
$\mathbb{C}^+$. It follows from \eqref{boundval}
that $L$ is locally in $L^2$ on $\mathbb{R}$. If $S$ has a pole $x$
on $\mathbb{R}$ it follows that $\frac{G(z)}{z-x}$ is in
$L^2(x-\delta, x+\delta)$ which contradicts Lemma  \ref{lem5}.
Assume, finally, that $Q$ has a zero in $\mathbb{C}^-$.
We have
$$
\Theta(z)\, \overline{G(\overline z)}\,
\overline{S(\overline z)} =
(\Theta(z) -1) \int \frac{\overline{G(t)}h(t) } {t-z} d\mu(t).
$$
The right-hand side is analytic in $\mathbb{C^+}$ whereas, by Lemma \ref{lem5},
$\Theta(z) \overline{G(\overline z)}$ is an outer function
in $\mathbb{C^+}$. Hence $\overline{S(\overline z)}$ is analytic in $\mathbb{C^+}$
and, thus, $S$ has no poles in $\mathbb{C}^-$. We conclude that $S$
is a polynomial.

Without loss of generality we may assume that $S\equiv 1$.
To complete the proof recall that
$S(t)G(t) = G(t) \overline{h(t)}$ $\mu$-a.e.
Put $E_1:=\{t: G(t) \ne 0\}$, $E_2:= \mathbb{R} \setminus E_1$.
We have $h(t)=1 $ $\mu$-a.e. on $E_1$. Since $h\in L^2(\mu)$, but
$\mu(\mathbb{R}) = +\infty$, we conclude that $\mu(E_1) <+\infty$
and $\mu(E_2) =+\infty$.
Then the functions $h\in K_\Theta$
such that $h=0$ $\mu$-a.e. on $E_1$ form an infinite dimensional subspace $X$ of
$K_\Theta$. For any $h\in X$ we have $Gh= 0$ $\mu$-a.e.
and hence \eqref{orthog} holds.
Since $X$ is infinite dimensional and contains $\frac{h(z)}{z-\lambda}$
whenever $\lambda \in \mathbb{C}^+$, $h(\lambda)=0$,
there exists a nonzero function $h_0\in X$ such that $(t+i)^N h_0 \in L^2(\mu)$.
Then for $L_0$ defined by \eqref{l} with $h_0$ in place of $h$ we have
$$
L_0(z) = (1-\Theta(z)) \int \frac{G(t)\overline{h_0(t)}}{t-z} d\mu(t)=
(1-\Theta(z)) \int \frac{G(t)}{(t+i)^N} \frac{(t+i)^N \overline{h_0(t)}}
{t-z} d\mu(t).
$$
The function $L_0$ vanishes at $\lambda_n$ and belongs
to $K_\Theta$ since it has a representation of
the form \eqref{clark}. Hence, $L_0 \equiv 0$ and $h_0 =0$ $\mu$-a.e.
This contradiction proves the theorem.
\end{proof}

\begin{remark}
{\rm If $G \in L^\infty(\mathbb{R})$, then the conditions $1 -
\Theta \notin L^2(\mathbb{R})$ and $\mu_{1}(\mathbb{R})=+\infty$ may
be omitted. Indeed, $p_\alpha =0$ for all unimodular $\alpha$ except
at most one, and the functions in $K_\Theta$ admit representation
\eqref{clark}. Then any function $L$ defined by \eqref{l} is in
$K_\Theta$ (note that $G\overline h \in L^2(\mu)$) and vanishes at
$\lambda_n$. Hence, $h\equiv 0$ and we conclude that the
biorthogonal system is complete. }
\end{remark}

\end{document}